\documentclass[10pt]{amsart}
\usepackage{amssymb}
\usepackage{listings}
\newcommand{\Ps}{\mathbf{P}}
\newcommand{\C}{\mathbf{C}}
\newcommand{\Q}{\mathbf{Q}}

\newtheorem{lemma}{Lemma}[section]
\newtheorem{proposition}[lemma]{Proposition}
\newtheorem{theorem}[lemma]{Theorem}

\theoremstyle{definition}

\newtheorem{definition}[lemma]{Definition}
\newtheorem{construction}[lemma]{Construction}
\newtheorem{example}[lemma]{Example}

\theoremstyle{remark}

\newtheorem{remark}[lemma]{Remark}

\DeclareMathOperator{\NL}{NL}
\DeclareMathOperator{\red}{red}
\DeclareMathOperator{\prim}{prim}
\DeclareMathOperator{\codim}{codim}

\title[Hodge loci of linear subvarieities]{Hodge loci associated with linear subspaces intersecting in codimension one}
\author{Remke Kloosterman}
\email{klooster@math.unipd.it}
\address{Universit\`a degli Studi di Padova,
Dipartimento di Matematica ``Tullio Levi-Civita",
Via Trieste 63,
35121 Padova, Italy}

\thanks{The author would like to thank Hossein Movasati for inviting him to IMPA in 2019. 
The author is a member of PRIN-2020 Project ``Curves, Ricci flat Varieties and their Interactions" and of INdAM-GNSAGA. This work was supported  by the BIRD-SID-2021 UniPD
project
 ``Moduli spaces of curves and hypersurface singularities".}
\date{\today}
\begin{document}
\begin{abstract}
Let $X\subset \Ps^{2k+1}$ be a smooth hypersurface containing two $k$-dimensional linear spaces $\Pi_1,\Pi_2$, such that $\dim \Pi_1\cap \Pi_2=k-1$.
In this paper we study the question whether the Hodge loci $\NL([\Pi_1]+\lambda[\Pi_2])$ and $\NL([\Pi_1],[\Pi_2])$ coincide. This turns out to be the case in a neighborhood of $X$ if $X$ is very general on $\NL([\Pi_1],[\Pi_2])$, $k>1$ and $\lambda\neq 0,1$. However, there exists a hypersurface $X$ for which $\NL([\Pi_1],[\Pi_2])$ is smooth at $X$, but $\NL([\Pi_1]+\lambda [\Pi_2])$ is singular for all $\lambda\neq0,1$. We expect that this is due to an embedded component of $\NL([\Pi_1]+\lambda[\Pi_2])$.
 The case $k=1$ was treated before by Dan, in that case $\NL([\Pi_1]+\lambda [\Pi_2])$ is nonreduced.

\end{abstract}
\maketitle

\section{introduction}

Let $k\geq 1$, $1\leq c \leq k+1$ and $d\geq 2+\frac{2}{k}$ be integers. 
Let $X \subset \Ps^{2k+1}$ be a hypersurface of degree $d$ containing two linear spaces $\Pi_1,\Pi_2$ of dimension $k$ such that $\dim \Pi_1\cap \Pi_2=k-c$. 

This paper is a second paper in a series concerning  the questions  for which $k,d,c$ and for which $\lambda\in \Q^*$ do we have $\NL([\Pi_1],[\Pi_2])=\NL([\Pi_1]+\lambda[\Pi_2])$ in a neighborhood of  $X$, and when do we have  $\NL([\Pi_1],[\Pi_2])=\NL([\Pi_1]+\lambda[\Pi_2])_{\red}$? 
In \cite[Chapters 18 and 19]{MovBook} Movasati considers  this  question. A first observation is that since $\NL([\Pi_1],[\Pi_2])$ is smooth at $X$ we have a positive answer to the first question if and only if $T_X\NL([\Pi_1],[\Pi_2])=T_X\NL([\Pi_1]+\lambda[\Pi_2])$.
This latter equality holds for all $X$ on  $\NL([\Pi_1],[\Pi_2])$ if $(c-1)(d-2)>2$ by \cite[Theorem 1.3]{RVLTs} (for the Fermat hypersurface) and \cite[Theorem 3.15]{RNKMov} (general case).

We now restrict to the case $(c-1)(d-2)\leq 2$.
In a previous paper \cite{RNKMov} we studied the cases where $c>1$. Then $(c,d)\in \{(2,4),(2,3),(3,3)\}$. We showed that if $k\geq 5$ (if $d=3)$ respectively $k\geq 2$ (if $d=4$) then there exists $X$ and $X'$  such that equality for the tangent spaces holds at $X$ for all but finitely many $\lambda$, but at $X'$ the locus $\NL([\Pi_1]+\lambda[\Pi_2])$ is reducible for all but finitely many $\lambda$.
On the other hand for $d=3, k\in \{3,4\}$ and $d=4,k=2$, we know that the two tangent spaces under consideration have different dimension for all $X$ and for all $\lambda$.

It remains to study the case $c=1$. If $\lambda=0$ then $\NL([\Pi_1]+\lambda [\Pi_2])$ paramterizes  hypersurfaces containing a $k$-plane and if $\lambda=1$ this locus parametrizes hypersurfaces containing a $k$-dimensional quadric. In both cases $\NL([\Pi_1]+\lambda[\Pi_2])$ is smooth and contains $\NL([\Pi_1],[\Pi_2])$ as a sublocus of positive codimension.

Now assume that $\lambda\in \Q\setminus\{0,1\}$.
 Maclean \cite{Maclean} considered the case $k=1,d=5$. She showed that the answer to the first question was negative, whereas the answer to the second question is affirmative, in particular the Hodge locus is nonreduced. Moreover, she showed that there are no further  nonreduced components of the Noether-Lefschetz locus for $k=1,d=5$. Dan \cite{DanNR} showed that if $k=1, d\geq 5, \lambda\in \Q \setminus\{0,1\}$ that $\NL([\Pi_1],[\Pi_2])\neq \NL([\Pi_1]+\lambda[\Pi_2])$ but $\NL([\Pi_1],[\Pi_2])=\NL([\Pi_1]+\lambda[\Pi_2])_{\red}$ in a neighborhood of $X$, i.e., that the locus is  nonreduced.
Recently, Duque and Villaflor \cite{DVjoin} announced a result which implies that $\NL([\Pi_1]+\lambda[\Pi_2])$ is singular at the Fermat hypersurface if $d\geq6$.

In this paper we focus on the case $c=1$. For $d=3,4$ the picture will be very similar to the case $c >1; (c-1)(d-2)\leq 2$, see Proposition~\ref{prpLowDeg}. However, for $d\geq 5$ the results differ significantly: For $k=1$ we will reprove Dan's result \cite{DanNR}  by different methods.
For $k\geq 2$ we will prove 

\begin{theorem} Let $d\geq 3$ be an integer.  Moreover,  suppose that $k>4$ if $d=3$, that $k>2$ if $d=4,5$ and $k\geq 2$ if $d\geq 6$.

Then there exists a hypersurface $X\subset \Ps^{2k+1}$ of degree $d$ containing two $k$-planes $\Pi_1,\Pi_2$ such that $\dim \Pi_1\cap \Pi_2=k-1$, and a finite set $\Sigma$ such that for all  $\lambda\in \Q\setminus \Sigma$ we have $\NL([\Pi_1]+\lambda[\Pi_2])=\NL([\Pi_1],[\Pi_2])$. Moreover, if $d\geq 6$ then we can take $\Sigma=\{0,1\}$.
\end{theorem}

Since the Zariski closure of $\NL([\Pi_1],[\Pi_2])$  in the locus of smooth degree $d$ hypersurfaces is  smooth and irreducible, the above theorem implies
that $\NL(\Pi_1+\lambda\Pi_2)$ is smooth and equals $\NL(\Pi_1,\Pi_2)$ for $k\geq 2$ in a neighborhood of $X$, provided that $X$ is sufficiently general and $\lambda\neq0,1$. However, as mentioned above, for special $X$ there is a difference in tangent space dimension and we expect that this is due to the fact that the  locus of split hypersurfaces yields an embedded component of $\NL([\Pi_1]+\lambda[\Pi_2])$.

The organization of this paper is as follows: In Section~\ref{secPrelim} we discuss the essential ingredients for our proof. For a more extensive discussion see \cite{RNKMov}. In Section~\ref{secResultA} we reprove Dan's result in Section~\ref{secResultB} we present our examples and prove the main result.

\section{Preliminaries}\label{secPrelim}
Fix an integers $k\geq 1, d\geq 2+\frac{2}{k}$. Let $S=\C[x_0,\dots,x_{2k+1}]$.
Let $X$ be a smooth hypersurface of degree $d$ in $\Ps^{2k+1}$. 
Let $J$ be the Jacobian ideal of $X$.

\begin{construction}\label{conIdeal}
Let $\gamma \in H^{k,k} (X,\C)\cap H^{2k}(X,\Q)$ be a Hodge class such that $\gamma_{\prim}$ is not zero, i.e., $\gamma$ is not a multiple of the class associated with the intersection of $X$ with $k$ hyperplanes.

 Griffiths' work \cite{GriRat} on the period map yields  an identification of $H^{2k-p,p}(X)_{\prim}$ with $(S/J)_{(p+1)d-2k-2}$.
Moreover, by Carlson-Griffiths \cite{CarGri} there is a choice of  isomorphisms $H^{2k,2k}(X)\cong \C\cong (S/J)_{(d-2) (2k+2)}$ such that
 the cupproduct \[ H^{2k-p,p}(X)_{\prim} \times H^{p,2k-p}_{\prim} \to H^{2k,2k}(X)\] equals the multiplication map
\[  (S/J)_{(p+1)d-2k-2} \times (S/J)_{(2k-p+1)d-2k-2} \to (S/J)_{(d-2)(2k+2)}.\]
Let $f_\gamma \in S_{(d-2)(k+1)}$ be the image of $\gamma_{\prim}$. Consider now the subspace  $f_\gamma^\perp$ in $(S/J)_{(d-2)(k-1)}$. Let $W$ be its lift to $S_{(d-2)(k+1)}$. 
Then $J_{(d-2)(k+1)}\subset W$.  Since $J_{d-1}$ is base point-free and $(d-2)(k+1)\geq d-1$ we have that $W$ is also base point free.
Let $I(\gamma)$ be the largest ideal of $S$ such that $I(\gamma)_{(k+1)(d-2)}=W$. 

Then $I(\gamma)$ is the \emph{ideal associated} with the Hodge class $\gamma$.
\end{construction}
\begin{remark}\label{rmkSocDeg}
By construction the algebra $S/I(\gamma)$ is an Artinian Gorenstein algebra of socle degree $(d-2)(k+1)$, see \cite[Remark 3.2]{RNKMov}.
\end{remark}
\begin{example}
Let  $\ell_0,\dots,\ell_k$ be linear forms, and let $g_0,\dots,g_k$ be forms of degree $d-1$, such that $\ell_0,\dots,\ell_k,g_0,\dots,g_k$ is a regular sequence. Let $f=\sum_{i=0}^k\ell_ig_i$ and $X=V(f)$. Suppose $X$ is smooth then the ideal associated with the Hodge class $\Pi=V(\ell_0,\dots,\ell_k)$ is
\[ \langle \ell_0,\dots,\ell_{k},g_0,\dots,g_k\rangle.\]
\end{example}

\begin{definition}
Suppose that $X\subset \Ps^{2k+1}$ is a smooth hypersurface  such that  $H^{k,k} (X,\C)_{\prim}\cap H^{2k}(X,\Q)_{\prim}$ is nonzero.
Let $\gamma \in H^{k,k} (X,\C)\cap H^{2k}(X,\Q)$ be a Hodge class such that $\gamma_{\prim}$ is nonzero.

Let $f\in \Ps(S_d)$ be such that $X=V(f)$.
Let $U$ be a small analytic  neighborhood of $[f]$ in $\Ps(S_d)$. Then we define \emph{the Hodge locus of $\gamma$}, denoted by $\NL(\gamma)$, as the set $X'\in U$ where under parallel transport the class $\gamma$ remains of type $(k,k)$ on $X'$.
\end{definition}
\begin{remark}
The definition of the Hodge locus $\NL(\gamma)$ depends on the choice of parallel transport, hence it is only well-defined if $U$ is simply connected. Moreover, the intersection of $U$ with the Zariski closure of $\NL(\gamma)$ in $U_d$ may be strictly larger than $\NL(\gamma)$.

The Hodge locus has a natural structure of an analytic scheme. Cattani-Deligne-Kaplan \cite{DelKap} showed that $\NL(\gamma)$ is actually algebraic.
\end{remark}

\begin{lemma}\label{lemCodimNL} Suppose that $k\geq 2$ and $d\neq 2+\frac{2}{k}$. Then we have
\[ \codim T_X\NL(\gamma) = \codim_{S_d} (I(\gamma))_d\]
\end{lemma}
\begin{proof}
See \cite[Lemma 3.6]{RNKMov}
\end{proof}

\begin{definition}\label{defCom}
Let $X\subset \Ps^{2k+1}$ be a smooth hypersurface of degree $d$. Let $U\subset U_d$ be a small neighborhood of $[X]$.
Pick  Hodge classes $\gamma_1,\dots,\gamma_r$ on $X$. Then the  \emph{Hodge locus of $\gamma_1,\dots,\gamma_r$}, denoted by $\NL(\gamma_1,\dots,\gamma_r)$, is the locus in $U$ where all $\gamma _i$ remain a Hodge class. 
\end{definition}
\begin{remark}\label{rmkIns}
From the above discussion it is immediate that
\[  T_X\NL(\gamma_1,\dots,\gamma_r)=\cap_{i=1}^r I(\gamma_i)_d=\cap_{i=1}^r T_X\NL(\gamma_i).\]
\end{remark}

The following result should be well-known to experts, a proof can be found in \cite[Proposition 3.9]{RNKMov}:
\begin{proposition}\label{prpSmooth} Let $X$ be a hypersurface of degree $d\geq 2+\frac{2}{k}$ containing two hyperplanes $\Pi_1,\Pi_2$. Then the Hodge locus $\NL(\Pi_1,\Pi_2)$ is smooth in a neighborhood of $X$.
\end{proposition}
\begin{remark} 
Note that if $X$ contains two pairs of $k$-planes intersecting in codimension $c$ then the Zariski closure of $\NL(\Pi_1,\Pi_2)$ is singular at $X$.
\end{remark}

\begin{definition}Let $X\subset \Ps^{2k+1}$ be a smooth hypersurface of degree $d$. Pick  Hodge classes $\gamma_1,\dots,\gamma_r$ on $X$.
Fix a point $(a_1:\dots :a_r)\in \Ps^{r-1}(\Q)$. 
We say that we have \emph{excess tangent dimension} at $(a_1:\dots:a_r)$ if
\[ T_X\NL(\gamma_1,\dots,\gamma_r) \subsetneq T_X\NL(a_1\gamma_1+\dots+a_r\gamma_r)\]
\end{definition}

\begin{lemma}\label{lemmakey}
Suppose $X\subset \Ps^{2k+1}$ is a smooth hypersurface of degree $d$ such that $X$ contains two subvarieties $Y_1,Y_2$ of dimension $k$. Suppose that $[Y_1]_{\prim}$ and $[Y_2]_{\prim}$ are linearly independent in $H^{2k}(X,\Q)_{\prim}$.

For $j=1,2$, let $I_j$ be the ideal associated with the Hodge class $[Y_j]$, cf. Construction~\ref{conIdeal}. Fix   isomorphisms $\sigma_j:(S/I_j)_{(k+1)(d-2)}\cong \C$. Let $\psi_j$ be the pairing
\[ (S/I_1\cap I_2)_d\times (S/I_1\cap I_2)_{kd-2k-2}\to  (S/I_1\cap I_2)_{(k+1)(d-2)} \to (S/I_j)_{(k+1)(d-2)}\stackrel{\sigma_j}{\longrightarrow}\C.\]

For $\lambda \in \Q^*$ let $I^{(\lambda)}$ be the ideal associated to $[Y_1]+\lambda [Y_2]$. Then there exists an injective function $\nu: \C^*\to\C^*$ such that
\[T_X\NL([Y_1]+\lambda[Y_2])/T_X\NL([Y_1],[Y_2])= I^{(\lambda)}_d/(I_1\cap I_2)_d=\ker_L (\psi_1+\nu(\lambda)\psi_2)\]
\end{lemma}

\begin{proof} See \cite[Lemma 3.12]{RNKMov}.
\end{proof}

\begin{theorem}\label{thmTsp}
Suppose $X\subset \Ps^{2k+1}$ is a smooth hypersurface of degree $d$ such that $X$ contains two subvarieties $Y_1,Y_2$ of dimension $k$. Suppose that $[Y_1]_{\prim}$ and $[Y_2]_{\prim}$ are linearly independent in $H^{2k}(X,\Q)_{\prim}$.

Let $I_j$ be the ideal associated with the Hodge class $[Y_j]$, cf. Construction~\ref{conIdeal}.
Suppose that the left kernel of the multiplication map
\[ (I_1+I_2/I_2)_d\times (I_1+I_2/I_2)_{kd-2k-2}\to (S/I_2)_{(k+1)(d-2)}\]
is zero. Then for all but finitely many $\lambda \in\Q^*$  we have
\[ \codim T_X\NL([Y_1]+\lambda [Y_2])=\codim T_X \NL([Y_1],[Y_2])\]
\end{theorem}
\begin{proof} See \cite[Theorem 3.13]{RNKMov}
\end{proof}

\section{Case $k=1$}\label{secResultA}
In the case $k\geq 2$ we will use Theorem~\ref{thmTsp}  to prove that there is no excess tangent space dimension for all but finitely many $\lambda$. However, in order to apply this Theorem successfully one needs that $d\leq kd-2k-2$ holds, i.e., if $k\geq 2$ then $d$ should be sufficiently large. However, for $k=1$ one cannot apply this result.
 Actually, as Dan \cite{DanNR} noted, also the conclusion of the Theorem does not hold. In this case one has that $\NL(\Pi_1+\lambda\Pi_2)$ is nonreduced for all but two values of $\lambda$:

\begin{proposition}[{Dan \cite{DanNR}}]\label{prpk1} Suppose $k=1$ and $d\geq 5$ then $\NL(\Pi_1+\lambda \Pi_2)$ is nonreduced for every $\lambda \in \Q\setminus \{0,1\}$.
\end{proposition}
\begin{proof}
In this case we may assume that $\Pi_1=V(x_0,x_3)$ and $\Pi_2=V(x_1,x_3)$. Then
$f=x_0x_1g+x_3h$, with $g\in S_{d-2}, h\in S_{d-1}$.
Since $f$ is smooth we have that $h(0,0,x_2,0)$ is a nonzero, i.e., $\alpha x_2^{d-2}$ for some $\alpha\in \C^*$.
In particular, we have 
\[I_1=\langle x_0,x_3,x_1g,h\rangle, \; I_2=\langle x_1,x_2,x_0g,h\rangle, \; I_1+I_2=\langle x_0,x_1,x_3,x_2^{d-1}\rangle.\]
From this it follows that 
\[ h_{I_1\cap I_2}(d-4)=h_{I_1}(d-4)+h_{I_2}(d-4)-h_{I_1}(d-4)=2d-7\]
 and $h_{I_1\cap I_2}(d)=2d-6$. In particular,  the pairing 
\[ S/(I_1\cap I_2)_{d}\times S/(I_1\cap I_2)_{d-4} \to \C\]
has a left factor which larger dimension then the right factor. Hence this pairing has  a nonzero left-kernel.
In particular $T_X\NL([\Pi_1]+\lambda[\Pi_2])/T_X\NL([\Pi_1],[\Pi_2])$ is nonzero for all $\lambda$. Since $\codim T_X\NL([\Pi_1],[\Pi_2])=h_{I_1\cap I_2}(d)=2d-6$ we have that the tangent space to $\NL([\Pi_1]+\lambda[\Pi_2])$ has codimension at most $2d-5$. 

Suppose now that $\NL([\Pi_1]+\lambda [\Pi_2])$ would have codimension less than $2d-5$ then by \cite{VoiCon} the class $([\Pi_1]+\lambda [\Pi_2])_{\prim}$ is a multiple of the primitive part of the class of a line or of a conic. This is only the case if $\lambda\in \{0,1\}$.
For the other values of $\lambda$ we have $\codim \NL([\Pi_1]+\lambda[\Pi_2])\geq 2d-6$, i.e., $\NL([\Pi_1]+\lambda [\Pi_2])_{\red}=\NL([\Pi_1],[\Pi_2])$.
\end{proof}

\section{Case $k\geq 2$}\label{secResultB}
We will now consider the case $k\geq 2$. We will show then for most choices of $d$ there is no excess tangent dimension for a general $X$ on $\NL(\Pi_1,\Pi_2)$. 
\begin{remark}
In \cite{RNKMov}  we noted that for $(c,d)\in\{(2,4),(3,3)\}$ and for almost all $\lambda\in \Q$ the locus $\NL([\Pi_1]+\lambda[\Pi_2])$ is reducible in the neighorbood of a so-called split hypersurface. We expect that for $c=1$ and $d\geq 5$ that the locus is  $\NL([\Pi_1]+\lambda[\Pi_2])$ contains an embedded component in the neighborhood of $X$ of split type. In the cases $d=3,4,c>1$ we used that in low dimension there is $\NL(\Pi_1,\Pi_2)$ is a proper subscheme of $\NL(\Pi_1+\lambda\Pi_2)_{\red}$ of positive codimension to show reducibility. In the case $c=1,d\geq 5,k=1$ we only have a strict inclusion on the level of tangent spaces, suggesting that there is a non-reduced component whose underlying subscheme is containg in $\NL([\Pi_1],[\Pi_2])$.
\end{remark}

\begin{proposition}\label{prpLowDeg} Suppose that $3\leq d\leq 5$. Moreover,  suppose that $k>4$ if $d=3$, that $k>2$ if $d=4,5$.

Then there exists a hypersurface $X\subset \Ps^{2k+1}$ of degree $d$ containing two $k$-planes $\Pi_1,\Pi_2$ such that $\dim \Pi_1\cap \Pi_2=k-1$, and such that for all but finitely many $\lambda$ we have $\NL([\Pi_1]+\lambda[\Pi_2])=\NL([\Pi_1],[\Pi_2])$.
\end{proposition}

\begin{proof}
We prove  by induction on $k$ that there is an $X$ on $\NL([\Pi_1],[\Pi_2])$ such that 
\[ T_X\NL([\Pi_1]+\lambda[\Pi_2])=T_X \NL([\Pi_1],[\Pi_2])\]
for all but finitely many $\lambda$. Since $\NL([\Pi_1],[\Pi_2])$ is smooth (Proposition~\ref{prpSmooth}) and $\NL([\Pi_1],[\Pi_2])\subset \NL([\Pi_1]+\lambda[\Pi_2])$ this suffices.

We take $X=V(f)$, where for $d\in \{4,5\},k=3$ we take
\[ f=x_0x_1\left(\sum_{i=0}^4 x_i^{d-2}\right)+\sum_{j=5}^7 x_j(x_{j-3}^{d-1}+x_j^{d-1}).\]
and for $d=3,k=5$ we take
\[ f=x_0x_1\left(\sum_{i=0}^6 x_i\right)+\sum_{j=7}^{11} x_j(x_{j-5}^{d-1}+x_j^{d-1}).\]
Then $\Pi_1=V(x_0,x_{k+2},\dots,x_{2k+1})$ and $\Pi_2=V(x_1,x_{k+2},\dots, x_{2k+1})$. 

One easily determines $I_1$ and $I_1+I_2$ in these cases and  checks by hand or by computer that the pairing
\[ (I_1+I_2/I_1)_d\times (I_1+I_2/I_1)_{kd-2k-2}\to (S/I)_{(k+1)(d-2)}\]
has no left kernel. Then Theorem~\ref{thmTsp} concludes the base case of the induction.
The induction step is \cite[Proposition 4.7]{RNKMov} combined with \cite[Lemma 2.10]{RNKMov}.
\end{proof}

To handle the case $c=1$, $d\geq 6, k\geq 2$ we have to slightly alter our approach.

For the rest of this section fix $k,d$ and let $X_{k,d}$ be the zero set of
\[ x_0x_1(x_0^{d-2}+x_1^{d-2}+x_2^{d-4}x_3^2)+x_4x_2^{d-1}+x_4^d+x_5x_3^{d-1}+x_5^d+\sum_{j=3}^k x_{2j+1}^d+x_{2j+1}x_{2j}^{d-1}\]
in $\Ps^{2k+1}$.
Then $X_{k,d}$ contains two $k$-planes
\[\Pi_1=V(x_0,x_4,x_5,x_7,x_9,x_{11},\dots,x_{2k+1})\]
and
\[\Pi_2=V(x_0,x_4,x_5,x_7,x_9,x_{11},\dots,x_{2k+1}).\]
Let $I_j$ be the ideal associated with the Hodge class $[\Pi_j]$ on $X_{k.d}$.

\begin{lemma} The set
\[B_1:= \{x_1^{a_1}x_2^{a_2}x_3^{a_3}\prod_{j=3}^k x_{2j}^{a_{2j}}\colon 0\leq a_i\leq d-2 \mbox{ for all } i\}\]
is a basis for $S/I_1$. Similarly, the set
\[ B_2:= \{x_0^{a_0}x_2^{a_2}x_3^{a_3}\prod_{j=3}^k x_{2j}^{a_{2j}}\colon 0\leq a_i\leq d-2 \mbox{ for all } i\}\]
is a basis for $S/I_2$.
\end{lemma}
\begin{proof}
The first statement follows from the fact that $I_1$ is generated by
\[ x_0,x_1^{d-1}+x_1x_2^{d-4}x_3^2,x_2^{d-1},x_3^{d-1},x_4,x_5,x_6^{d-1},x_7,x_8^{d-1},x_9,\dots,x_{2k}^{d-1},x_{2k+1}.\]
The second statement follows similarly.
\end{proof}

\begin{lemma} The set $B_1\cup B_2$ is a basis for $S/I_1\cap I_2$.
\end{lemma}
\begin{proof}
One easily checks that $B_1\cap B_2$ is a basis for $S/I_1+I_2$. Hence
\begin{eqnarray*}\# B_1\cup B_2&=&\#B_1+\# B_2-\# B_1\cap B_2\\&=&\dim S/I_1+\dim S/I_2-\dim S/I_1+I_2=\dim S/I_1\cap I_2.\end{eqnarray*}
Moreover, $B_1\cup B_2$ generates $S/I_1\cap I_2$, hence it is a basis.
\end{proof}

Let $M_1=x_1^{d-2}x_2^{d-2}x_3^{d-2}\prod_{i=3}^k x_{2i}^{d-2}$, let $M_2=x_0^{d-2}x_2^{d-2}x_3^{d-2}\prod_{i=3}^k x_{2i}^{d-2}$. Then for $j=1,2$ the monomial $M_j$ is a basis for the one-dimensional vector space $(R/I_j)_{(k+1)(d-2)}$.

 \begin{lemma} Let $j\in \{1,2\}$. Let $N=x_{2-j}^{a_{2-j}}x_2^{a_2}x_3^{a_3}\prod_{i=3}^k x_{2i}^{a_{2i}}\in B_j$ be a monomial of degree $t$. If $a_j<d-2$, $a_2>2$ or $a_3>d-4$ then there is a unique monomial $N'\in B_j$ of degree $(k+1)(d-2)-t$ such that $\varphi_j(N,N')\neq 0$. Moreover, $N'=M_j/N$.
If $a_{2-j}=d-2$, $a_2\leq 2$ and $a_3\leq d-4$ then there exists two such monomials namely $N'=M_j/N$ and $N''=M_j/N x_{2-j}^{d-2}/x_2^{d-4}x_3^2$.
\end{lemma}

\begin{proof}Using symmetry, it suffices to prove the statement for $j=1$.

Take integers  $b_i\leq d-2$ for $i=1,2,3,6,8,10,\dots,2k$ such that $\sum b_i=(k+1)(d-2)-t$ and consider the product
\[  \left(x_1^{a_1}x_2^{a_2}x_3^{a+3}\prod_{i=3}^k x_{2i}^{a_{2i}}\right)\left(x_1^{b_1}x_2^{b_2}x_3^{b_3}\prod_{i=3}^k x_{2i}^{b_{2i}}\right)\]
If $a_1+b_1\leq d-2$ then this monomial is nonzero if and only if $a_i+b_i\leq d-2$ for all $i$.
Since $\sum a_i+b_i= (k+1)(d-2)$ this happens if and only if $\sum a_i+b_i=d-2$ for all $i$, i.e., when
\[ x_1^{b_1}x_2^{b_2}x_3^{b_3}\prod_{i=3}^k x_{2i}^{b_{2i}}=\frac{M_1}{ x_1^{a_1}x_2^{a_2}x_3^{a+3}\prod_{i=3}^k x_{2i}^{a_{2i}}}.\]

If $a_1+b_1\geq d-2$ then $a_1+b_1\leq 2(d-2)$ and we find
\[  x_1^{a_1}x_2^{a_2}x_3^{a+3}\prod_{j=3}^k x_{2j}^{a_{2j}}x_1^{b_1}x_2^{b_2}x_3^{b_3}\prod_{j=3}^k x_{2j}^{b_{2j}}\] is equivalent to \[ -x_1^{a_1+b_1-d+2}x_2^{a_2+b_2+d-4}x_3^{a_3+b_3+2}  \prod_{j=3}^k x_{2j}^{a_{2j}+b_{2j}}\] moduli $I_1$.
The latter is nonzero in $(S/I_1\cap I_2)_{(k+1)(d-2)}$ if and only if $a_1=b_1=(d-2), a_2+b_2=2, a_3+b_3=d-4, a_{2j}+b_{2j}=d-2$.
\end{proof}
\begin{remark}
If $t=d$ then $a_1=d-2$ implies $a_2\leq 2, a_3\leq d-4$. (Here we use $d\geq 6$.)
\end{remark}

\begin{lemma}\label{lemGram} Suppose $\nu(\lambda)\neq 0$. Let $\psi=\varphi_1+\nu(\lambda)\varphi_2$ restricted to
\[ (S/I_1\cap I_2)_d \times (S/I_1\cap I_2)_{kd-2k-2}\]
Let $C_1$ be subset of $B_1\cup B_2$   consisting of the mononmials of of degree $d$, let $C_2$ be the subset of $B_1\cup B_2$ consisting of monomials of degree $kd-2k-2$. Then up to permutation of elements of $C_1$ and $C_2$ we have that the Gram matrix of $\psi$ with respect to $C_1$ and $C_2$ has no zero rows and is a block matrix with blocks
\[ (1),\; \begin{pmatrix} \nu(\lambda)\end{pmatrix}, \; \begin{pmatrix}1 & \nu(\lambda)\end{pmatrix},\;\begin{pmatrix}
 0& 1 & \nu(\lambda)\\
 1& -1 &0\\
 \nu(\lambda) &0 &-\nu(\lambda)\\ \end{pmatrix}\]
 and each of these blocks occur.
 \end{lemma}
 \begin{proof}
 Let  $N\in C_1$ be a monomial of degree $d$. Then at least one of  $a_0,a_1$ is zero.

If $0<a_1<d-2$ then let $N'=M_1/N$. Then  $\psi(N,N'')=0$ for all  monomials $N''$ of degree $kd-2k-2$ different from $N'$ and $\psi(N''',N')=0$ for all  monomials $N'''$ of degree $d$ different from $N$. Moreover, $\psi(N,N'')\in \{1,\nu(\lambda)\}$.
Similarly if $0<a_0<d-2$ we can take $N'=M_2/N$ and have the same properties. This yields the two types of $1\times 1$-blocks.

If $a_0=a_1=0$ then $\psi(N,N')$ is nonzero if and only if $N'=M_1/N$ or $N'=M_2/N$. If $N'$ and $N''$ do not pair with another monomial $N''$ of degree $d$ then this yields a block of the third type.
If one of the $N'$ pairs nonzero with another monomial $N''$ of degree $d$ then $a_2\geq d-4$ and $a_3\geq  2$. In this case $N''=Nx_1^{d-2}/x_2^{d-4}x_3^2$ or $N''=Nx_0^{d-2}/x_2^{d-4}x_3^2$, i.e., the exponent of $x_0$ or of $x_1$ in $N''$ equals $d-2$.
If $N$ is a monomial of degree $d$ with $a_1=d-2$ then $\psi(N,N')$ is nonzero for $N'=M_1/N$ and $N'=M_1/N x_1^{d-2}/x_2^{d-4}x_3^2$. In this case $\psi(N,N')\in \{\pm 1\}$.
If $N$ is a monomial of degree $d$ with $a_0=d-2$ then $\psi(N,N')=\pm \nu(\lambda)$ is nonzero for $N'=M_2/N$ and $N'=M_2/N x_1^{d-2}/x_2^{d-4}x_3^2$. In this case $\psi(N,N')\in \{\pm \nu(\lambda)\}$.
These monomials yield a $3\times 3$-block
\[\begin{pmatrix}
 0& 1 & \nu(\lambda)\\
 1& -1 &0\\
 \nu(\lambda) &0 &-\nu(\lambda)\\
 \end{pmatrix}.
 \]
 \end{proof}

\begin{theorem}\label{thmHighDeg} Let $d\geq 6, k\geq 2$ be integers. Let $\lambda\in \Q\setminus \{0,1\}$. Then $\NL([\Pi_1]+\lambda [\Pi_2])=\NL([\Pi_1],[\Pi_2)]$ in a neighborhood of $X_{k,d}$.
\end{theorem}

\begin{proof}
Recall that $\NL([\Pi_1],[\Pi_2])\subset \NL([\Pi_1]+\lambda [\Pi_2])$ and $T_{X_{k,d}}\NL([\Pi_1],[\Pi_2])\subset T_{X_{k,d}} \NL([\Pi_1]+\lambda [\Pi_2])$. Moreover, the quotient of the two tangent spaces equals the left kernel of
\[ (S/I_1\cap I_2)_d \times (S/I_1\cap I_2)_{kd-2k-2}\to (S/I_1)_{(k+1)(d-2)}/(M_1-\nu(\lambda) M_2)\stackrel{\sim}{\longrightarrow} \C\]
by Lemma~\ref{lemmakey}.
From Lemma~\ref{lemGram} it follows that the associated Gram-matrix is a block matrix whose blocks are square with determinant, $1,\nu(\lambda),\nu(\lambda)(\nu(\lambda)+1)$ or have full rowrank. Hence if $\nu(\lambda)\not  \in \{0,-1\}$ we have that the matrix has full row rank and that there is no left kernel.

However, if $\nu(\lambda)\in\{0,-1\}$ then the $3\times 3$-block has zero determinant, yielding a nontrivial left kernel.
Hence there are two values of $\lambda$ for which there is a left kernel. It is well-known that for $\lambda=0$ and $\lambda=1$ this is the case. This finishes the proof.
\end{proof}

\begin{remark}
In the remaining cases $d=3,k=3,4$, $d=4,k=2$ and $d=5,k=2$ there is always excess tangent dimension, hence our method does not apply. For these case the main question is whether this happens because $\dim \NL([\Pi_1]+\lambda[\Pi_2])>\dim \NL([\Pi_1],[\Pi_2])$ (as for $d=4,k=1$ or $d=5,k=2$) or that this locus is nonreduced (as happens for $c=1,f\geq 5,\lambda\neq1$).
\end{remark}

\bibliographystyle{plain}
\bibliography{remke2}

\end{document}